\documentclass{amsart}
\usepackage{amssymb}
\usepackage{eucal}
\usepackage{amsfonts}
\usepackage{epsfig}
\usepackage{color}
\usepackage[all]{xy}
\usepackage[pdftex]{hyperref}
\let\oldcite\cite                                  

\newtheorem{thm}{Theorem}[section]
\newtheorem{cor}[thm]{Corollary}
\newtheorem{lem}[thm]{Lemma}
\newtheorem{prop}[thm]{Proposition}
\theoremstyle{definition}
\newtheorem{defn}[thm]{Definition}
\theoremstyle{remark}

\numberwithin{equation}{section} \theoremstyle{remark}
\newtheorem{ex}[thm]{Example}

\newcommand{\X}{\mathcal{X}}
\newcommand{\Y}{\mathcal{Y}}
\newcommand{\Z}{\mathcal{Z}}
\newcommand{\PP}{\mathcal{P}}

\newcommand{\F}{\mathcal{F}}

\newcommand{\B}{\mathcal{B}}
\newcommand{\A}{\mathcal{A}}

\newcommand{\HH}{\mathcal{H}}
\newcommand{\G}{\mathcal{G}}

\newcommand{\bbX}{\mathbb{X}}

\newcommand{\tX}{\tilde{\X}}

\let\:=\colon

\newcommand{\be}{\beta}
\newcommand{\ep}{\epsilon}

\newcommand{\Ra}{\Rightarrow}

\newcommand{\sfT}{\mathsf{T}}
\newcommand{\Top}{\mathsf{Top}}
\newcommand{\CGTop}{\mathsf{CGTop}}

\newcommand{\x}{\times}

\newcommand{\Hom}{\operatorname{Hom}}

\newcommand{\Map}{\operatorname{Map}}

\newcommand{\hFib}{\operatorname{hFib}}
\newcommand{\Fib}{\operatorname{Fib}}
\newcommand{\Tor}{\operatorname{Tor}}

\newcommand{\id}{\operatorname{id}}
\newcommand{\pr}{\operatorname{pr}}
\newcommand{\ev}{\operatorname{ev}}

\def\smashedst{\setbox0=\hbox{$\rightrightarrows$}\ht0=4pt\box0}
\newcommand{\sst}[1]{\stackrel{#1}{\smashedst}}

\def\twomorphism{\setbox0=\hbox{$\Rightarrow$}\ht0=4pt\box0}
\newcommand{\twomor}[1]{\stackrel{#1}{\twomorphism}}

\newcommand{\arf}[6]{\ar@{}@<#3>#2 | (#4){#6}="#5" \ar#1#2  }
\newcommand{\ars}[4]{\arf{#1}{#2}{#3}{#4}{s}{}}
\newcommand{\art}[4]{\arf{#1}{#2}{#3}{#4}{t}{.}}

\begin{document}
\title{Fibrations of topological stacks}
\author{Behrang Noohi}

\begin{abstract}
 In this note we define fibrations of topological stacks and establish their main
 properties. We prove various standard results about fibrations 
 (fiber homotopy exact sequence, Leray-Serre and Eilenberg-Moore
 spectral sequences, etc.). We
 prove various criteria for a morphism of topological stacks to be a fibration,
 and use these to produce  examples of  fibrations. We prove that
 every morphism of topological stacks factors through a fibration 
 and construct the homotopy fiber
 of a morphism of topological stacks.  When restricted to 
 topological spaces our notion of fibration coincides with the classical one.
\end{abstract}

\maketitle

\tableofcontents

 \section{Introduction}{\label{S:Introduction}}

In this note we define fibrations of topological stacks and establish their main
properties. Our theory generalizes the classical theory in the sense that
when restricted to ordinary topological spaces our notion(s) of fibration 
reduce to the usual one(s) for topological spaces. 

We prove some standard results for fibrations of topological stacks: the
fiber homotopy exact sequence (Theorem \ref{T:fhes}), the Leray-Serre 
spectral sequence (Theorems \ref{T:HSS} and \ref{T:CSS}), the Eilenberg-Moore
spectral sequence (Theorem \ref{T:EM}), and so on.
We also prove that every morphism of topological stacks factors
through a fibration (Theorem \ref{T:replacement}). We use this
to define the homotopy fiber $\hFib(f)$ of a morphism $f \: \X \to \Y$ of  
stacks (see $\S$\ref{S:Factor}). 

Since our fibrations are not assumed to be representable, it is often
not easy to check wether a given morphism of stacks is a fibration straight
from the definition. We prove the following useful local criterion for
a map $f$ to be a weak Serre fibration;\footnote{
For most practical purposes (e.g, constructing spectral sequneces or the fiber 
homotopty exact sequence) weak fibrations are as good as fibrations.} see Theorem \ref{T:strong}.

\begin{thm}
   Let $p \: \X \to \Y$ be a morphism of topological stacks.
   If $p$ is   locally a  weak Hurewicz fibration then it is a
   weak Serre fibration.
\end{thm}

In $\S$\ref{S:Examples} we provide some general classes of examples of fibrations of stacks.
Throughout the paper, we also prove various results which can be used to produce new fibrations out
of old ones. This way we can produce plenty of examples of fibrations of topological stacks.

The results of this paper, 
though of more or less of standard nature, are not straightforward. The difficulty
being that colimits are not  well-behaved in the 2-category of topological stacks; for instance, gluing
continuous maps along closed subsets is not always possible.
For this reason, the usual methods for proving lifting properties
(which involve building up continuous maps by
attaching cells or extending maps from smaller subsets by inductive steps)  
often can not be applied to stacks. 

The main technical input which allows us to circumvent this difficulty  
is the theory of classifying spaces for topological stacks developed in
\cite{Homotopytypes} combined with Dold's results on fibrations \cite{Dold}.

To reduce the burden of terminology we have opted to  state the results of this paper only for 
topological stacks. There is, however, a more general class of stacks, called paratopological
stacks, to which these results can be generalized -- the proofs will be identical. 

Paratopological stacks have a major advantage over topological stacks: they accommodate a
larger class of mapping stacks (see \cite{Mapping}, Theorem 1.1), while essentially enjoying all the 
important features of topological stacks. This is especially important as mapping stacks provide
a wealth of examples of fibrations ($\S$\ref{SS:mappingfiberations}). The reader can consult
\cite{Homotopytypes} and \cite{Mapping} for more on paratopological stacks.

 \section{Review of homotopy theory of topological stacks}{\label{S:Homotopy}}

In this section, we  recall some  basic facts and definitions from  
\cite{Foundations} and \cite{Homotopytypes}. For a quick introduction to topological stacks the 
reader may also consult \cite{Rapid}. Our terminology is different from that of
\cite{Foundations} in that what is called a {\em pretopological} stack in [ibid.] is called a
{\em topological} stack here.

 \subsection{Topological stacks}{\label{SS:topological}}

To fix our theory of topological stacks, the first thing to do is to fix the base Grothendieck
site $\sfT$. The main two candidates are the category $\Top$ of all topological spaces (with the
open-cover topology) and the category $\CGTop$ of  compactly generated (Hausdorff) 
topological spaces (with the open-cover topology). 
The latter behaves better for the purpose of having a theory of fibrations (as in
\cite{Whitehead}), so throughout the text our base Grothendieck site will be $\sfT:=\CGTop$.

By a {\em topological stack} we mean a  category $\X$ fibered in groupoids
over $\sfT$ which is equivalent to the quotient stack of a topological groupoid
$\bbX=[R\sst{}X]$, with $R$ and $X$  
topological spaces. The reader who is not comfortable with fibered categories may pretend that
$\X$ is a presheaf of groupoids over $\sfT$.

Topological stacks form a 2-category which is closed under 2-fiber products. 
For every two topological stacks $\X$ and
$\Y$, a {\em morphism} $f \: \X \to \Y$ between them is simply a morphism of the
underlying fibered categories (or underlying
presheaves of groupoids, if you wish). Given morphisms  $f,g \: \X \to \Y$, a 
{\em 2-morphism} $\varphi \: f \twomor{} g$ is a natural transformation relative to $\sfT$.
If $h \: \Y \to \Z$ is another morphism of stacks, we denote the induced 2-morphism
$h\circ  f \twomor{} h\circ g$ by $h\circ\varphi$ or $h\varphi$.
We use the multiplicative notation  $\varphi\psi$ for the composition of 2-morphisms
$\varphi \: f \twomor{} g$ and $\psi \: g \twomor{} h$. 

Via Yoneda we identify $\sfT$ with a full subcategory of the 2-category of topological stacks.
A morphism $f \: \X \to \Y$ of topological stacks is called {\em representable} if for every
morphism $Y \to \Y$ from a topological space $Y$, the fiber product $Y\times_{\Y}\X$ is
(equivalent to) a topological space. It turns out that, for every topological stack $\X$,
every morphism $f \: X \to \X$, with $X$ a topological space, is representable.
 
 \begin{defn}{\label{D:HurSerre1}}
   A topological stack $\X$ is {\bf Hurewicz} 
  (respectively, {\bf Serre}) if it admits a presentation $\bbX=[R\sst{}X]$ 
  by a topological groupoid in which the source (hence also the target) map 
  $s \: R \to X$ is  locally (on source and
  target) a  Hurewicz (respectively, Serre) fibration.
  That is, for every $a \in R$, there exists an open
  neighborhood $U\subseteq R$ of $a$  and $V\subseteq X$ of $f(a)$ such
  that the restriction of $s|_U \: U \to V$ is a Hurewicz (respectively, Serre)  fibration.
 \end{defn}
 
Hurewicz (respectively, Serre) topological stacks from a full sub 2-category of the 2-category
of topological stacks which is closed under 2-fiber products and contains the category of
topological spaces.
 
 \begin{defn}{\label{D:HurSerre2}}
   Let $f \: \X \to \Y$  be a morphism of topological stacks. We say that $f$ is {\bf Hurewicz} 
   (respectively, {\bf Serre}) if for every map $T \to \Y$ from a topological space $T$, the fiber
   product $T\times_{\Y}\X$ is a Hurewicz (respectively, Serre) topological stack
   (Definition \ref{D:HurSerre1}).
 \end{defn}

 \begin{lem}{\label{L:HurSerre1}}
    Every representable morphism of topological stacks is Hurewicz and every
    Hurewicz morphism of topological stacks is Serre.
 \end{lem}

 \begin{proof}
  Trivial.
 \end{proof}

 \begin{lem}{\label{L:HurSerre2}}
    Let $f \: \X \to \Y$ be a morphism of topological stacks. If $\X$ is 
    Hurewicz (respectively, Serre)  then $f$ is Hurewicz (respectively, Serre).
 \end{lem}

 \begin{proof}
   Let $X \to \X$ be a chart for $\X$ such that the corresponding groupoid
   presentation $[R\sst{} X]$, $R=X\times_{\X} X$, satisfies the condition of Definition 
   \ref{D:HurSerre1}. Then, for every topological space $T$ mapping to $\Y$,
   the pullback groupoid  $[R_T\sst{} X_T]$ is
   a presentation for $\X_T=T\times_{\Y}\X$ which satisfies the condition of Definition 
   \ref{D:HurSerre1}. Here, $X_T=T\times_{\Y}X$ and $R_T=T\times_{\Y}R$.
 \end{proof}

\subsection{Homotopy between maps}{\label{SS:homotopy}}

 \begin{defn}{\label{D:relhtpy}}
   Let $p \: \X \to \Y$ and $f,g \: \A \to \X$ be morphisms of topological stacks. Let
   $\varphi \: p\circ f \twomor{} p\circ g$ be a 2-morphism
        $$\xymatrix@R=28pt@C=44pt@M=5pt{ & \X \ar[d]^p \\
           \A \ars{@/^/}{[r]}{1ex}{0.5}  ^(0.6){ p\circ f}
              \art{@/_/}{[r]}{-1ex}{0.5}   _(0.6){ p\circ g}
                \ar[ur]^{f,g}         
             \ar    @{=>}_{\varphi} "s";"t"   
              &  \Y  }$$     
   A {\bf homotopy} from $f$ to $g$ relative to $\varphi$ is a quadruple
   $(H,\ep_0,\ep_1,\psi)$ as follows: 

  \begin{itemize}

    \item  A map  $H \: I\times \A\to \X$, where $I$ stands for the interval $[0,1]$.

    \item  A pair of 2-morphisms $\ep_0 \: f \twomor{} H_0$ and
           $\ep_1 \: H_1 \twomor{} g$.  Here $H_0$ and $H_1$ stand for
           the maps $\A \to \Y$ obtained
           by restricting $H$ to $\{0\}\x\A$ and $\{1\}\x\A$, respectively.
     \item A 2-morphism $\psi \: p\circ \tilde{f} \twomor{} p\circ H$ such that
     $\psi_0=p\circ\ep_0$ and $(\psi_1)(p\circ\ep_1)=\varphi$.  
     Here $\tilde{f} \: I\times \A \to \X$    stands for $f\circ\pr_2$. 
  \end{itemize}
    We usually drop $\ep_0$, $\ep_1$ and $\psi$ from the notation. A  {\bf 
    ghost homotopy}  from $f$ to $g$ is a  2-morphism $\Phi \: f \twomor{} g$
    such that $p\circ\Phi=\varphi$. Equivalently,
    this means that  $H$ can be chosen to be  (2-isomorphic to) a constant homotopy
   (namely, one factoring through the projection $\pr_2\: I\times\A \to \A$).
 \end{defn}

The usual notion of homotopy (and ghost homotopy) between maps $f, g \: \A \to \X$
corresponds to the case where $\Y$ is a point. Note that in this case the
2-morphisms $\varphi$ and $\psi$ are necessarily the identity 2-morphisms.

As discussed in \cite{Foundations}, $\S$17 (also see $\S$\ref{SS:pushout} below), 
gluing homotopies can in general be problematics
unless we make certain fibrancy conditions on the target stack. There is, however, a way to glue
homotopies which is well-defined up to a homotopy of homotopies and works for arbitrary topological 
stacks. 

 \begin{lem}{\label{L:glue}}
   Let $p \: \X \to \Y$ be a morphism of topological stacks.
   Let $\A$ be a topological stack, and consider morphisms
   $f_1,f_2,f_3 \: \A \to \Y$,  and 2-morphisms
   $\varphi_{12} \: p\circ f_1 \twomor{} p\circ f_2$ 
   and $\varphi_{23} \: p\circ f_2 \twomor{} p\circ f_3$. If $f_1$ is homotopic to
   $f_2$ relative to $\varphi_{12}$, and  $f_2$ is homotopic to
   $f_3$ relative to $\varphi_{23}$, then $f_1$ is homotopic to
   $f_3$ relative to $\varphi_{12}\varphi_{23}$. 
 \end{lem}

 \begin{proof}
  Let $(H_{12},\ep_0,\ep_1,\psi)$
  be a homotopy from $f_1$ to $f_2$ relative to $\varphi_{12}$. Similarly, let
  $(H_{23},\delta_0,\delta_1,\chi)$ be
  a homotopy from $f_2$ to $f_3$ relative to $\varphi_{23}$.
  We construct a  homotopy $H_{13}$ from $f_1$ to $f_3$ relative to 
  $\varphi_{12}\varphi_{23}$. 
  The point of the following rather unusual construction  is that gluing maps along closed subsets 
  may not be possible, but we can always glue maps along open subsets. 
   
  Define $\tilde{H}_{12} \:  [0,2/3)\times \A\to \X$ by 
   $$\tilde{H}_{12}(t,a) =\left\{
      \begin{array}{ll}
     H_{12}(3t,a), & t \leq 1/3 \\
      H_{12}(1,a), & 1/3 \leq t < 2/3
   \end{array} \right.$$
  More precisely, $\tilde{H}_{12}=H_{12}\circ R$, where $R=(r,\id_{\A}) \: 
  [0,2/3)\times \A \to  I\times \A$ is defined by
       $$r(t) =\left\{
      \begin{array}{ll}
     3t, & t \leq 1/3 \\
      1, & 1/3 \leq t < 2/3
   \end{array}\right.$$
  Similarly, we define  $\tilde{H}_{23} \:  (1/3,1]\times \A\to \X$ by
   $$\tilde{H}_{23}(t,a) =\left\{
      \begin{array}{ll}
     H_{23}(0,a), & 1/3 < t \leq 2/3 \\
      H_{23}(3t-2,a), & 2/3 \leq t \leq 1
   \end{array}\right.$$
  We glue $\tilde{H}_{12}$ to $\tilde{H}_{23}$  along the open set 
  $(1/3,2/3)\times \A$ using the 2-isomorphism 
   $$\Xi\: 
  \tilde{H}_{12}|_{(1/3,2/3)} \twomor{}  \tilde{H}_{23}|_{(1/3,2/3)}, \ \ 
  \Xi=\tilde{\ep}_1\tilde{\delta}_0,$$ 
  where $\tilde{\ep}_1=\ep_1\circ \pr_2$, with
  $\pr_2 \: (1/3,2/3)\times \A \to \A$ being the second projection  
  ($\tilde{\delta}_0$ is defined similarly). Denote the  resulting glued map by
  $H \: [0,1]\times \A \to \X$. The quadruple
  $(H,\ep_0,\delta_1,\psi)$ provides the desired homotopy from $f_1$ to
  $f_3$ relative to $\varphi_{12}\varphi_{23}$.  
  We leave it to the reader to verify that the axioms of Definition \ref{D:relhtpy} are satisfied.
 \end{proof}

\subsection{Pushouts in the category of topological stacks}{\label{SS:pushout}}

A  downside of the 2-category of topological stacks is that in it pushouts are not 
well-behaved. For example, let $T$ be a topological space which is a union of two
subspaces $B$ and $C$, and let $A = B\cap C$. If $B$ and $C$ are open, then 
for any stack $\X$, any two morphisms $B \to \X$ and $C \to \X$ which agree along $A$
can be glued to a morphism from $T$ to $\X$ (this is simply the descent condition). If, however,
$B$ and $C$ are not open, one may not, in general, be able to glue such overlapping morphisms.
As we will see below in Proposition \ref{P:pushout1}, this is partly remedied if we 
impose a cofibrancy condition
on the inclusions $A \subset B$ and $A \subset C$ and  a fibrancy condition
on $\X$.

 \begin{prop}{\label{P:pushout1}}
  Let  $i \: A \hookrightarrow B$ and $j \: A \hookrightarrow C$ be
  embeddings of topological spaces which are (locally)  cofibrations. Then, the pushout
  $B\vee_A C$ remains a pushout in the 2-category of
  Hurewicz topological stacks (Definition \ref{D:HurSerre1}). That is,
  for every Hurewicz topological stack $\X$, the diagram
   $$\xymatrix@M=6pt{  \Hom(B\vee_A C,\X) \ar[r] \ar[d] &   \Hom(C,\X) \ar[d] \\
      \Hom(B,\X) \ar[r] &  \Hom(A,\X)}$$
  is 2-cartesian.  (The arrows in the diagram are the obvious restriction maps.)
 \end{prop}

 \begin{proof}
  Follows from  (\cite{Foundations}, Theorem 16.2).
 \end{proof}

Proposition \ref{P:pushout1} may fail if one of the maps $i$ or
$j$ is not an embedding, or if $\X$ is an arbitrary topological stack. 
For instance, even when $(Y,A)$ is a nice pair
(say an inclusion of a finite CW complex into another), the quotient
space $Y/A$ may not in general have the universal property of a
quotient space when viewed in the category of (Hurewicz) topological stacks.

Things, however, work much better up to homotopy, as we see in Proposition \ref{P:pushout2} below. 
To state the proposition, we make the following definition.
Let  $f \: A \to B$ and $g \: A \to C$ be continuous
maps of topological spaces. We define $B\bar{\vee}_A C$ to be 
  $$B\bar{\vee}_A C:=B\vee_{A\times\{0\}}(A\times I)\vee_{A\times\{1\}}C.$$
In other words,  $B\bar{\vee}_A C$ is obtained by 
gluing the two ends of the tube $A\times I$ to $B$ and $C$
using $f$ and $g$. For a topological stack $\X$, we define the restricted hom
$\Hom^*(B\bar{\vee}_A C,\X)$ via the following 2-fiber product diagram
      $$\xymatrix@M=6pt{  \Hom^*(B\bar{\vee}_A C,\X) \ar[r] \ar[d] &   
         \Hom(B\bar{\vee}_A C,\X) \ar[d]^r \\
               \Hom(A,\X) \ar[r]_c & \Hom(A\times I,\X)   }$$ 
Here, $r$ is induced by the inclusion $A\times I \hookrightarrow B\bar{\vee}_A C$, and 
$c$ is induced by the projection $A\times I \to A$. In simple terms,
the restricted hom   consists of maps $B\bar{\vee}_A C \to \X$ which are
constant along the tube $A\times I$. If $\X=X$ is an honest topological space,
then $\Hom^*(B\bar{\vee}_A C,X)=\Hom(B\vee_A C,X)$.

 \begin{prop}{\label{P:pushout2}}
   Let  $f \: A \to B$ and $g \: A \to C$
   be continuous maps of topological spaces. Then,
   for every topological stack $\X$ the diagram
     $$\xymatrix@M=6pt{  \Hom^*(B\bar{\vee}_A C,\X) \ar[r] \ar[d] &   \Hom(C,\X) \ar[d] \\
        \Hom(B,\X) \ar[r] &  \Hom(A,\X)}$$
  of groupoids is 2-cartesian. (The arrows in the diagram are the obvious restriction maps.)
 \end{prop}

 \begin{proof}
  The proof of this proposition is very similar to the proof of Lemma \ref{L:glue}. We sketch the idea.
  We want to show that  two maps  $u \: B \to \X$ and  $v \: C \to \X$ which are
  identified by a 2-isomorphism along $A$ give rise to a map $w\: B\bar{\vee}_A C\to\X$
  which is constant along the tube $A\times I$.
  The map $u \: B \to \X$ gives rise to a map 
     $$\tilde{u} \: B\vee_{A\times\{0\}}(A\times [0,1)) \to \X,$$
  $\tilde{u}:=p\circ u$, where $p \: B\vee_{A\times\{0\}}(A\times [0,1)) \to B$ is the map
  which is the identity on $B$ and is  $f\circ\pr_1$ on $A\times [0,1)$. Similarly,
  $v \: C \to \X$ gives rise to a map 
      $$\tilde{v} \:  (A\times (0,1])\vee_{A\times\{1\}}C\to \X.$$
  The two maps $\tilde{u}$ and $\tilde{v}$ are equal along the open subset
  $A\times (0,1)$ of $B\bar{\vee}_A C$, so they glue to a map $w\: B\bar{\vee}_A C \to\X$.
 \end{proof}
 \subsection{Shrinkable morphisms of stacks}{\label{SS:shrinkable}}
 
 \begin{defn}{\label{D:retraction}}
   We say that a morphism  $f \: \X \to \Y$ of  stacks is
   {\bf shrinkable} if it admits a section $s \: \Y \to \X$ (meaning that $f\circ s$ is 2-isomorphic
   to $\id_{\Y}$ via some $\varphi \:  \id_{\Y} \Ra f\circ s$)
   such that there is  a homotopy from $\id_{\X}$ to $s\circ f$ relative to $\varphi\circ f$
   (Definition \ref{D:relhtpy}). 

   We say that $f$ is   {\bf locally shrinkable}
   if there is an epimorphism $\Y' \to \Y$ such
   that the base extension $f' \: \X' \to \Y'$ of $f$ to $\Y'$ is shrinkable.
 \end{defn}

In the case where $f \: X \to Y$ is a continuous map of topological spaces,
this definition coincides with the one in \cite{Dold}, $\S$1.5. In this case, $s$
identifies $Y$ with the closed subspace $s(Y)$ of $X$, and there is a
fiberwise strong deformation retraction of $X$ onto $s(Y)$. Note that, for
a general morphism of stacks $f \: \X \to \Y$, a section $s$ of $f$ may not be an embedding.
(For example, the map $[*/G] \to *$ has a section which is not an embedding,
namely, the projection $* \to [*/G]$. As another example, take the map 
$[\mathbb{R}/\mathbb{Q}] \to *$.)  All we can say in general
is that $s$ is representable.
 
 \begin{lem}{\label{L:locallyshrinkable1}}
  Let $f \: \X \to \Y$ be a  morphism of stacks and  $s \: \Y \to \X$ a section for it.
  Suppose that $f$ is shrinkable onto $s$.  Let $g \: \Y' \to \Y$ be an arbitrary morphism. 
  Then, the base extension $f' \: \X' \to \Y'$ of $f$ along $g$ is shrinkable onto
  $s'$.
          $$\xymatrix@M=6pt{\X' \ar[d]^{f'} \ar[r] & \X \ar[d]^f \\
              \Y' \ar[r]_{g}  \ar@/^/ @{..>} [u]^{s'} & \Y \ar@/^/ @{..>} [u]^{s}}$$ 
  where $s'$ is the section induced by $s$.     
  In particular, any base extension of a (locally) shrinkable morphism is (locally) shrinkable.
 \end{lem}
 
 \begin{proof}
  Easy.
 \end{proof} 

 \begin{lem}{\label{L:locallyshrinkable2}}
  A continuous map $f \: X \to Y$ of topological spaces is locally shrinkable if and only if 
  there is an open cover $\{U_i\}$ of $Y$ such
  that $f|_{U_i} \: f^{-1}(U_i) \to U_i$ is shrinkable for all $i$.
  A representable morphism $f \: \X \to \Y$ of stacks is locally shrinkable
  if for any map $T \to \Y$ from a topological space $T$, the base extension
  $f_T \: \X\times_{Y}T \to T$ is locally shrinkable. 
 \end{lem}

 \begin{proof}
  Easy.
 \end{proof}

 \begin{lem}{\label{L:locallyshrinkable3}}
    If $f \: \X \to \Y$ is a locally shrinkable representable morphism of stacks, then
    $f$ is a universal weak equivalence. That is,  
    for any map $T \to \Y$ from a topological space $T$, the base extension
    $f_T \: \X\times_{Y}T \to T$ is a weak equivalence of topological spaces.
 \end{lem}

 \begin{proof}
  This is \cite{Homotopytypes}, Lemma 5.4.
 \end{proof}

The following lemma provides a natural class of shrinkable morphisms  
constructed using mapping stacks. 
For a discussion of mapping stacks we refer the reader to \cite{Mapping}.
All we need to know here about mapping stacks
$\Map(\X,\Y)$ is the functoriality in the two variables $\X$ and $\Y$ and the exponential
property. 
 
 \begin{lem}{\label{L:contractible}}
   Let $X$ be a topological space and $r \: I\times X  \to X$
   a deformation retraction of $X$ onto a point $x \in X$. Let $\Y$ be a stack, and
   let $c_{\Y} \: \Y \to \Map(X,\Y)$ be the morphism parametrizing the
   constant maps from $X$ to $\Y$ (more precisely,  $c_{\Y}$ is induced from
   $X \to *$ by the functoriality of the mapping stack). Let $\ev_x \: 
   \Map(X,\Y) \to \Y$ be the evaluation map at $x$. Then, $c_{\Y}$
   is  a section of $\ev_x$ and $\ev_x$ is shrinkable onto  $c_{\Y}$.
 \end{lem}

 \begin{proof}
   The deformation retraction $r \: I\times X  \to X$ induces
      $$r^* \: \Map(X,\Y) \to \Map(I\times X, \Y)\cong \Map(I,\Map(X,\Y)).$$
   Using the exponential property again, this gives the desired shrinking map
      $$H \:  I \times \Map(X,\Y) \to \Map(X,\Y).$$
    That is, $H$ is fiberwise homotopy from $\id_{\Map(X,\Y)}$ to $c_{\Y}\circ\ev_x$.
 \end{proof} 
  
\subsection{Classifying spaces of topological stacks}{\label{SS:classifying}}  
 
The main theorem here is the following. It will be of crucial importance for us throughout the paper.

 \begin{thm}[\oldcite{Homotopytypes}, Theorem 6.3]{\label{T:nice}}
  Every topological stack $\X$ admits an atlas $\varphi \: X \to \X$
  which is  locally shrinkable. In particular, $\varphi$ is a universal weak equivalence.
 \end{thm}

By definition (\cite{Homotopytypes}, Definition 6.2), such an atlas $\varphi \: X \to \X$ is a 
{\em classifying space} for $\X$. A classifying space $X$ captures the homotopy theoretic 
information in $\X$
via the map $\varphi$. In the following sections we will make use of classifying spaces to reduce
problems about topological stacks to ones about topological spaces.

\section{Fibrations between stacks}{\label{S:Fibrations1}}

 \begin{defn}{\label{D:lifting}}
  Let $i \: \A \to \B$ and $p \: \X \to \Y$ be morphisms
  of stacks.  We say that $i$ has {\bf weak left lifting property}
  (WLLP) with respect to
   $p$, if for every 2-commutative
  diagram
           $$\xymatrix@=8pt@M=6pt{
         \A   \ar[rr]^f\ar[dd]_{i}   & \ar@{=>}[dl]^{\alpha}&
              \X  \ar[dd]^{p}  \\ & & \\
         \B    \ar[rr]_g
        & &  \Y    }$$
  we can find a morphism $k \: \B \to \X$, a 2-morphism $\be \: p\circ k \Ra g$,
  and a fiberwise homotopy  (Definition \ref{D:relhtpy})
  $H$ from $f$ to  $k\circ i$  relative to  $\be\circ i$,  as in the diagram
      $$\xymatrix@=8pt@M=6pt{
         \A   \ar[rr]^f\ar[dd]_{i}   & &
              \X  \ar[dd]^{p}  \\  \ar@{}[ur] |-{H} & & \ar@{:>}[dl]_{\be}\\ %
         \B \ar@{..>}[uurr] |-{k}   \ar[rr]_g
        & &  \Y    }$$
  (So $p\circ H$ is a ghost homotopy  from
  $p\circ f$ to  $p\circ k\circ i$  and we have
  $(p\circ H )(\be\circ i)=\alpha \: p\circ f \Ra g\circ i$.)
  If the homotopy $H$ can be chosen to be a ghost homotopy
  (Definition \ref{D:relhtpy}), we say that $i$
  has {\bf  left lifting property} (LLP) with respect to $p$.
 \end{defn}

 \begin{defn}[\oldcite{Dold} $\S$5]{\label{D:WCHP}}
   Let $\A$ be a stack and $p \: \X \to \Y$  a morphism of stacks.
   We say that $p$ has {\bf weak covering homotopy
   property} (WCHP)  for $\A$ if the inclusion $i \: \A \to \A\times I$, $a\mapsto (a,0)$,
   has  WLLP with respect to $p$. Given a fixed class
   $\mathcal{T}$ of stacks (e.g., all compactly generated 
   topological spaces, paracompact spaces,
   CW complexes, etc.) we say that $p$ has WCHP with respect to $\mathcal{T}$ if
   it has WCHP for all $\A \in \mathcal{T}$.
   Similar definitions can be made for {\bf covering homotopy  property} (CHP).
 \end{defn}

 \begin{lem}{\label{L:lifting}}
  Consider the 2-commutative diagram of stacks
     $$\xymatrix@=8pt@M=6pt{
         \A  \ar[rr]\ar[dd]_{i}  & \ar@{=>}[dl] & \X' \ar[dd]_{p'} \ar[rr] &
                               \ar@{=>}[dl]  & \X  \ar[dd]^{p} \\
                          & &  & &   \\
         \B   \ar[rr] & &  \Y' \ar[rr] && \Y   }$$
  Assume that the right square is 2-cartesian. Then, the (weak) left lifting problem
  of $i$  can be solved with respect to $p$ if and only if  it can be solved with respect to $p'$.
 \end{lem}

 \begin{proof}
  Straightforward.
 \end{proof}

 \begin{lem}{\label{L:compose0}}
   Let $p \: \X \to \Y$, $q \: \Y \to \Z$ and $i \: \A \to \B$ be
   morphisms of  stacks. If $i$ has
   LLP with respect to $p$ and $q$, then it has LLP with respect to $q\circ p$.
   If $i$ has WLLP  with respect to $p$ and $q$, and $p$ has
   WCHP with respect to $\A$,  then $i$ has WLLP  with respect to $q\circ p$.  
 \end{lem}

 \begin{proof}
  The lemma is straightforward without the `weak' adjective.
  The case where the adjective `weak' is present is less trivial, so we give more details.
  Consider the weak lifting problem
        $$\xymatrix@=8pt@M=6pt{
         \A   \ar[rr]^f\ar[dd]_{i}   & \ar@{=>}[dl] &
              \X  \ar[dd]^{q\circ p}  \\ & & \\
         \B    \ar[rr]_g
        & &  \Z    }$$
  We solve it in three steps.  First we solve
       $$\xymatrix@=8pt@M=6pt{
          \A  \ar[rr]^{p\circ f} \ar[dd]_i & & \Y  \ar[dd]^{q}    \\
         \ar@{}[ur] |-{h} & & \ar@{:>}[dl] \\ %
         \B \ar@{..>}[uurr] |-{k}   \ar[rr]_g
        & &  \Z    }$$
  Here, $h$ is a fiberwise homotopy from $p\circ f$ to $k\circ i$. Next, we solve
        $$\xymatrix@=8pt@M=6pt{
          \A  \ar[rr]^f \ar[dd]_{0\times\id_A} & & \X  \ar[dd]^{p}    \\
         \ar@{}[ur] |-{h'} & & \ar@{:>}[dl]\\ %
         I\times \A \ar@{..>}[uurr] |-{l}   \ar[rr]_h
        & &  \Y    }$$
  Here, $h'$ is a fiberwise homotopy from $f$ to $l|_{t=0}$.      Finally, we solve
          $$\xymatrix@=8pt@M=6pt{
          \A  \ar[rr]^{l|_{t=1}} \ar[dd]_i & & \X  \ar[dd]^{p}    \\
         \ar@{}[ur] |-{h''} & & \ar@{:>}[dl] \\ %
         \B \ar@{..>}[uurr] |-{m}   \ar[rr]_k 
        & &  \Y    }$$
  Here, $h''$ is a fiberwise homotopy from  $l|_{t=1}$   to $m\circ i$. The solution to our
  original problem would then  be
         $$\xymatrix@=8pt@M=6pt{
          \A  \ar[rr]^{f} \ar[dd]_i & & \X  \ar[dd]^{q\circ p}    \\
         \ar@{}[ur] |-{H} & & \ar@{:>}[dl] \\ %
         \B \ar@{..>}[uurr] |-{m}   \ar[rr]_g 
        & &  \Z    }$$
  where $H$ is the fiberwise homotopy from $f$ to  $m\circ i$ obtained by gluing $h'$, $l$, and
  $h''$ (Lemma \ref{L:glue}).
 \end{proof}

 \begin{lem}{\label{L:compose1}}
  Let $p \: \X \to \Y$, $q \: \Y \to \Z$ and $i \: \A \to \B$ be
  morphisms of  stacks. If $\emptyset \to \A$ has LLP with respect to
  $p$ and $i$ has LLP (resp., WLLP) with respect to
  $q\circ p$, then $i$ has LLP (resp., WLLP) with respect to $q$.
 \end{lem}  
 
 \begin{proof}
   Straightforward.
 \end{proof} 

 \begin{defn}{\label{D:fibration}}
  Let $p \: \X \to \Y$ be a morphism of  stacks. We say that
  $p$ is a {\bf Hurewicz fibration}, if it has CHP for
  all compactly generated topological spaces.
  We say that $p$ is a {\bf Serre fibration}, if it has CHP for all finite CW
  complexes. (In general, we can define a {\em $\mathcal{T}$-fibration} to
  be a map which has CHP for $\mathcal{T}$.) Similarly, we can
  define {\em weak Hurewicz fibration} and {\em weak Serre
  fibration} (more generally, {\em weak $\mathcal{T}$-fibration}).
 \end{defn}

 \begin{defn}{\label{D:trivial}}
  Let $p \: \X \to \Y$ be a representable  morphism of
  stacks. We say that
  $p$ is a {\bf (weak) trivial Hurewicz fibration}, if every cofibration
  $ i \: A \to B$ of   topological spaces has (W)LLP with respect to
  $p$. We say that $p$ is  {\bf (weak) trivial Serre fibration} if every
  cellular inclusion  $ i \: A \to B$ of CW complexes has (W)LLP with respect
  to $p$.
 \end{defn}

 \begin{lem}{\label{L:compose2}}
  Let $p \: \X \to \Y$ and $q \: \Y \to \Z$  be
  morphisms of  stacks. If $p$ and $q$ are
  (weak) (trivial) Hurewicz/Serre fibrations, then so is $q\circ p$.
 \end{lem}

 \begin{proof}
  Follows from Lemma \ref{L:compose0}.
 \end{proof}

 \begin{defn}{\label{D:locally}}
  Let $P$ be a property of morphisms of  stacks
  which is invariant under base extension (e.g., any of the properties
  defined in Definitions \ref{D:fibration},  \ref{D:trivial}).
  We say that a morphism $p \: \X \to \Y$ is stacks is
  {\bf locally} $P$, if for some epimorphism $\Y' \to \Y$
  the base extension $p' \: \Y'\times_{\Y}\X \to \Y'$ is $P$.
 \end{defn}

 \begin{lem}{\label{L:epibasechange}}
  Let $P$ be a property of morphisms of stacks which is
  invariant under base extension. For example, $P$ can be
  any of the following: (weak) (trivial) Hurewicz/Serre/$\mathcal{T}$ fibration. Let
  $p \: \X \to \Y$ be a morphism of  stacks and
  $\Y' \to \Y$ an epimorphism. If the base extension
  $p' \: \Y'\times_{\Y}\X \to \Y'$ is   locally $P$, then so is $p$.
 \end{lem}

 \begin{proof}
  Straightforward.
 \end{proof}

 \begin{lem}{\label{L:basechange1}}
  Let $p \: \X \to \Y$ be a morphism of  stacks.  
  Let $P$ be any of the properties (locally) (weak) (trivial) Hurewicz/Serre
  fibration (or, $\mathcal{T}$-fibration, for a class 
  $\mathcal{T}$ of topological spaces). If $p$ is $P$, then the base
  extension of $p$ along any morphism $\Y' \to \Y$ of 
  stacks is again $P$. Conversely, if the base extension of $p$ along
  every  morphism $B \to \Y$, with $B$ a topological space,
  is $P$, then $p$ is $P$. (In the case of (weak) $\mathcal{T}$-fibrations
  it is enough to take $B$ in  $\mathcal{T}$.)
 \end{lem}

 \begin{proof}
  The proof is a simple application of Lemma \ref{L:lifting}. Here is how the typical
  argument works. Let
  $i \: A \to B$ be a map for which we want to prove  (W)LLP. To see that $i$
  has (W)LLP with respect to $p$, apply Lemma \ref{L:lifting} to the
  following diagram
    $$\xymatrix@R=8pt@C=2pt@M=6pt{
         A  \ar[rr]\ar[dd]_{i}  & \ar@{=>}[dl] & B\times_{\Y}\X \ar[dd]_{p_B} \ar[rr] &
                               \ar@{=>}[dl]  & \X  \ar[dd]^{p} \\
                          & &  & &   \\
         B   \ar[rr]_{\id} & &  B \ar[rr] && \Y   }$$
 \end{proof}

If  a continuous map $p \: X \to Y$ of topological spaces
is a trivial Hurewicz fibration then $p$ is  shrinkable.
The converse is not true, but we have the following.

 \begin{prop}{\label{P:weaktrivial}}
   Let $p \: X \to Y$  be a continuous map of topological spaces. Then, the following
   are equivalent:
   \begin{itemize}
     \item[1)] Every continuous map $i \: A \to B$ has WLLP with respect to $p$;
     \item[2)]  The map $p$ is a weak trivial Hurewicz fibration (Definition \ref{D:trivial});
     \item[3)] The map $p$ is a weak Hurewicz fibration
   and a homotopy equivalence;
     \item[4)] The map $p$ is shrinkable.
   \end{itemize}
 \end{prop}
 
 \begin{proof} The implication 1) $\Rightarrow$ 2) is  obvious.
  
 \medskip 
 \noindent{\em Proof of 2) $\Rightarrow$ 3).} Let $\operatorname{Cyl}(p)=(X\times[0,1])\coprod_p Y$ 
 be the mapping cylinder
 of $p$, and consider the lifting problem
          $$\xymatrix@M=6pt{
         X  \ar[r]^{\id_X} \ar[d]_{i}   &   
              X  \ar[d]^{p}  \\   
         \operatorname{Cyl}(p)    \ar[r]_g  \ar@{..>}[ru]^{k}
        &   Y    }$$
  Here, $i \: X \to \operatorname{Cyl}(p)$ is the natural inclusion of $X$ in $\operatorname{Cyl}(p)$,
  and $g \: \operatorname{Cyl}(p) \to Y$ is defined by $g(x,t)=p(x)$ and $g(y)=y$, for $x \in X$
  and $y \in Y$. Since $i$ is a  Hurewicz cofibration, we can find  a lift $k$ in the diagram. If we set
  $q=k|_Y \: Y \to X$, it follows that $p\circ q =\id_Y$ and $q\circ p$ is homotopic to
  $\id_X$. This proves that $p$ is a homotopy equivalence. 
  
  To see that $p$ is a weak Hurewicz fibration, observe that the  $t=0$ inclusion 
  $A \to A\times [0,1]$ is a Hurewicz cofibration. 
 
 \medskip 
 \noindent{\em Proof of 3) $\Rightarrow$ 4).} This is \oldcite{Dold}, Corollary 6.2.
 
 \medskip 
 \noindent{\em Proof of 4) $\Rightarrow$ 1).} Consider
  a   lifting problem
          $$\xymatrix@M=6pt{
         A  \ar[r]^{f} \ar[d]_{i}   &   
              X  \ar[d]^{p}  \\   
         B    \ar[r]_g \ar@{..>}[ru]^{k}
        &   Y    }$$
  where $i$ is an arbitrary continuous map. Since $p$ is shrinkable, it admits a section
  $s \: Y \to X$. Set $k=s\circ g$. This makes the lower triangle commutative. The fiberwise
  deformation retraction of $X$ onto $s(X)$ provides a fiberwise homotopy between $f$ and
  $k\circ i$.
 \end{proof}

 \begin{cor}{\label{C:strong}}
  A representable locally shrinkable map $p \: \X \to \Y$
  of stacks is  locally a weak trivial Hurewicz fibration (Definition \ref{D:locally}).
 \end{cor}

 \begin{proof}
  By Lemma \ref{L:basechange1}, we are reduced to the case where $X$ and $Y$ are spaces. 
  The claim follows from  Proposition \ref{P:weaktrivial}.
 \end{proof}
 
 \begin{lem}{\label{L:Dold2}}
   Let $p \: X \to Y$ be  a locally shrinkable map of topological spaces. If $Y$ is
   paracompact, then $p$ is 
   shrinkable. In particular, every map $\emptyset \to A$, 
   with $A$ a paracompact topological space,  has 
   LLP with respect to every locally shrinkable 
   representable morphism $p \: \X \to \Y$ of topological stacks.
 \end{lem}
 
 \begin{proof}
      Follows from \cite{Dold}, $\S$2.1.
 \end{proof}

 \begin{prop}[Hurewicz Uniformization Theorem]{\label{P:uniformization}}
  Let $p \: X \to Y$ be  locally a (weak) (trivial) Hurewicz fibration
  of topological spaces, with $Y$  paracompact.
  Then,  $p$ is a (weak) (trivial) Hurewicz fibration.
 \end{prop}

 \begin{proof}
  Without the `trivial' adjective this follows from  \cite{Dold}, Theorems 4.8 and 5.12. 
  If $p$ is  locally a weak trivial Hurewicz fibration, then it is
  locally shrinkable (Proposition \ref{P:weaktrivial}), hence shrinkable (Lemma \ref{L:Dold2}),
  and so a  weak trivial Hurewicz fibration  (Proposition \ref{P:weaktrivial}).
 
  If $p$ is  locally a  trivial Hurewicz fibration, then, as we 
  explained in the beginning of the proof, it follows from Dold's result that
  $p$ is a Hurewicz fibration . On the other hand,
  we just showed that $p$ is a weak trivial Hurewicz fibrations. Therefore, by
  Proposition \ref{P:weaktrivial}, $p$ is a homotopy equivalence. This proves that
  $p$ is a  trivial Hurewicz fibration.
 \end{proof}

In practice, it is much easier to check that a given morphism of stacks is {\em locally} a 
fibration. However, for applications (such as the fiber homotopy exact sequence) we need
to have a global fibration (at least a weak one). Proposition \ref{P:strong2} and Theorem 
\ref{T:strong}  provide local criteria for a morphism of stacks to be a (weak)
Serre fibration.

 \begin{prop}{\label{P:strong2}}
   Let $p \: \X \to \Y$ be a representable morphism of topological stacks.
   If $p$ is   locally a  (weak) (trivial) Hurewicz fibration then it is a
   (weak) (trivial) Serre fibration. 
 \end{prop}

 \begin{proof}
   By Lemma \ref{L:basechange1}, it is enough to prove the statement after base extending
   $p$ along an arbitrary map $Y \to \Y$ from a topological space $Y$.
   So we are reduced to the case
   where $p \: X \to Y$ is a continuous map of topological spaces.

  To show that
  $p$ is a (weak) (trivial) Serre fibration it is enough to check that the base
  extension $p_K \: K\times_{Y}X \to K$ of $p$ along any morphism
  $g \: K \to Y$, with $K$  a
  finite CW complex, is a (weak) (trivial) Serre fibration. Since being locally a
  (weak) Hurewicz fibration is invariant under base change,
  $p_K$ is  locally a (weak) (trivial) Hurewicz fibration.
  Since $K$ is paracompact, $p_K$ is indeed a (weak) (trivial) Hurewicz fibration
  (Proposition \ref{P:uniformization}), hence also  a  (weak) (trivial) Serre fibration.
 \end{proof}

 \begin{cor}{\label{C:weakSerre}}
  Let $\X$ be a topological stack and $\varphi \: X \to \X$ a classifying space for $\X$ 
  as in Theorem \ref{T:nice}. Then $\varphi$ is a weak trivial Serre fibration. 
 \end{cor}

 \begin{lem}{\label{L:compose3}} 
  Let $p \: X \to \Y$ and $q \: \Y \to \Z$ be
  morphisms of topological stacks, with $X$  a topological space. 
  Suppose that $p$ is locally shrinkable and $q\circ p$ is a weak (trivial)
  Serre fibration. Then $q$ is a weak (trivial)
  Serre fibration.
 \end{lem}

 \begin{proof}
  Follows from Lemmas \ref{L:compose1} and \ref{L:Dold2}.
 \end{proof}

We now come to the main result of this section. It strengthens Proposition  
\ref{P:strong2} by removing the representability condition on $p$, at the cost of
having to add the adjective `weak'. 
 
 \begin{thm}{\label{T:strong}}
   Let $p \: \X \to \Y$ be an arbitrary morphism of topological stacks.
   If $p$ is   locally a  weak (trivial) Hurewicz fibration then it is a
   weak (trivial) Serre fibration.
 \end{thm}

 \begin{proof}
  Let $\varphi \: X \to \X$ be a classifying space for $\X$ as in Theorem \ref{T:nice}.
  Since $\varphi$ is locally shrinkable, it is  locally a weak trivial Hurewicz fibration 
  (Proposition \ref{P:weaktrivial}). By Lemma \ref{L:compose2}, $p\circ\varphi$ is also 
  locally a weak (trivial) Hurewicz fibration. By Proposition \ref{P:strong2}, $p\circ\varphi$
  is a weak (trivial) Serre fibration. The claim follows from Lemma \ref{L:compose3}.
 \end{proof}

\subsection{Lifting property of fibrations with respect to cofibrations}{\label{SS:lifting}}

 \begin{prop}{\label{P:lifting}}
  Let $p \: \X \to \Y$ be a morphism of topological stacks. Suppose that $p$ is 
  a Hurewicz (respectively, Serre) morphism (Definition \ref{D:HurSerre2}). 
  If $p$  is a Hurewicz (respectively, Serre)
  fibration, then every trivial cofibration (respectively, 
  cellular inclusion of finite CW complexes
  inducing isomorphisms on all $\pi_n$)
  $i \: A \to B$ has LLP with respect to $p$. Here, by a trivial cofibration we mean
  a DR pair $(B,A)$ of   topological spaces in the sense of \cite{Whitehead}.
 \end{prop}

 \begin{proof}
  Using the usual base extension trick, we may assume that $\Y=Y$ is a topological space and
  $\X$ is a Hurewicz (respectively, Serre) topological stack. The proof of Theorem 7.16 of 
  \cite{Whitehead} now applies.
 \end{proof}

 \begin{prop}{\label{P:fibration}}
   Let $p \: \X \to \Y$ be a  morphism of topological stacks.
   Suppose that $p$ is 
   a Hurewicz (respectively, Serre) morphism (Definition \ref{D:HurSerre2})
   and a  Hurewicz (respectively, Serre) fibration. Then, 
   for every cofibration $i \: A \to B$ of  
   topological spaces, every homotopy lifting extension problem
         $$\xymatrix@R=12pt@C=10pt@M=6pt{
               (I\times A) \cup (\{0\}\times B)   \ar[rr]^(0.7)f\ar@{^(->}[dd] & 
                     \ar@/^/@{=>}[dl]_{\alpha}  &   \X  \ar[dd]^{p}  \\ & & \\
                     I\times B   \ar[rr]_g  \ar@{..>}[uurr]_h       & &  \Y    }$$
   has a solution  $h \: I \times B \to \X$.  
 \end{prop}  

 \begin{proof}
   Use Proposition \ref{P:lifting} and repeat the proof of \cite{Whitehead}, Theorem 7.16.
 \end{proof}

 \begin{cor}{\label{C:weakTriv}}
  Let $p \: \X \to \Y$ be a
  Hurewicz (respectively, Serre)   morphism of topological stacks. If $p$ is 
  a  Hurewicz (respectively, Serre) fibration and a weak trivial 
  Hurewicz (respectively, Serre)  fibration (Definition \ref{D:trivial}), then it is a trivial 
  Hurewicz (respectively, Serre)  fibration.
 \end{cor}

 \begin{proof}
  Consider a lifting problem
      $$\xymatrix@=8pt@M=6pt{
         A   \ar[rr]^f\ar[dd]_{i}   & \ar@{=>}[dl] &
              \X  \ar[dd]^{p}  \\ & & \\
         B  \ar@{..>}[uurr]_k  \ar[rr]_g
        & &  \Y    }$$
  Since $p$ is a weak trivial fibration, there is a lift $k \:   B \to \X$ together 
  with a  fiberwise homotopy $H$ from   $f':=k\circ i$  to $f$. By Proposition
  \ref{P:fibration}, we can solve the homotopy lifting extension problem
         $$\xymatrix@R=12pt@C=10pt@M=6pt{
               (I\times A) \cup (\{0\}\times B)   \ar[rr]^(0.7){H\cup k}\ar[dd] & 
                     \ar@/^/@{=>}[dl] &   \X  \ar[dd]^{p}  \\ & & \\
                     I\times B   \ar[rr]_{g\circ\pr_2}  \ar@{..>}[uurr]_h       & &  \Y    }$$  
  Letting  $f:=h|_{\{1\}\times B}$ we find a solution to our original lifting problem.        
 \end{proof}

\section{Examples of fibrations}{\label{S:Examples}}

In this section we discuss a few general classes of examples of fibrations of stacks.
Of course, a trivial class of examples is the ones coming from topological spaces,
because any notion of fibration we have defined for morphisms of stacks $f \: \X \to \Y$
coincides with the corresponding classical notion when $\X$ and $\Y$ are topological spaces. 

Another tirival class of examples is the projections maps.

 \begin{lem}{\label{L:projection}}
  Let $\X$ and $\Y$ be arbitrary stacks. Then, the projection $\X \times \Y \to \X$ is
  a Hurewicz fibration.
 \end{lem}

 \begin{proof}
  Trivial.
 \end{proof}

\subsection{Fibrations induced by mapping stacks}{\label{SS:mappingfiberations}}

 \begin{prop}{\label{P:Whitehead1}}
    Let $A \to B$ be a cofibration of  topological spaces, and let
    $\X$ be a  stack. Then the induced morphism 
        $$\Map(B,\X) \to \Map(A,\X)$$
    is a Hurewicz fibration.
 \end{prop}

 \begin{proof}
  This is Theorem 7.8 of \cite{Whitehead}. The same proof works.
 \end{proof}

 \begin{ex}{\label{E:fib}}
  Let $L\X$ and $P\X$ be the loop and path stacks of a stack $\X$. 
  Then, the time $t$ evaluation maps $\ev_t \: L\X \to \X$ and $\ev_t \: P\X \to \X$ are Hurewicz 
  fibrations. Also, the map $(\ev_0,\ev_1) \: P\X \to \X\times\X$ is a Hurewicz fibration.
 \end{ex}

\begin{prop}{\label{P:Whitehead2}}
  Let $p \: \X \to \Y$ be a (weak) (trivial) Hurewicz (respectively, Serre) fibration of stacks. 
  Then, for any topological space (respectively, CW complex) $Z$  the induced map
       $$\Map(Z,\X) \to \Map(Z,\Y)$$
  is a (weak) (trivial) Hurewicz (respectively, Serre) fibration. 
\end{prop}

 \begin{proof}
 This is Theorem 7.10 of \cite{Whitehead}. The same proof works.
 \end{proof}

\subsection{Quotient stacks}{\label{SS:quotient}}

Let $\X$ be a topological stack and $[R\sst{}X]$ a topological groupoid presentation
for it. If $s \: R \to X$ is a (weak) (trivial) Hurewicz/Serre fibration, then  the quotient
map $p \: X \to \X$ is locally a (weak) (trivial) Hurewicz/Serre fibration. This follows from
the 2-cartesian diagram
  $$ \xymatrix@M=6pt{ R \ar[r]^t \ar[d]_s  & X \ar[d]^p \\
         X \ar[r]_p & \X}$$
and Lemma \ref{L:epibasechange}. 

In particular, by Proposition \ref{P:uniformization},
if $s$ is a (weak) (trivial) Hurewicz fibration, then $p \: X \to \X$ 
is a (weak) (trivial) Serre fibration whose base extension
along any map $T \to \X$ from a paracompact $T$ is a (weak) (trivial) Hurewicz fibration.

 \begin{ex}{\label{E:quotient}}
   The above discussion applies to the case $\X=[X/G]$, where $G$ is a topological group
   acting on a topological space $X$. Therefore, the quotient map $p \: X \to [X/G]$ is
   Serre fibration (with fiber $G$) which becomes a Hurewicz fibration upon base extension
   along any map $T \to [X/G]$ from a paracompact $T$. If the group $G$ is contractible,
   then $p$ is a trivial Serre fibration whose base extension $p_T$ is a trivial Hurewicz fibration for
   $T$ paracompact.
 \end{ex}
 
\subsection{Covering maps}{\label{SS:covering}} 
Let $p \: \X \to \Y$ be a covering morphism of stacks in the sense of  \cite{Foundations}, 
$\S$18. Then $p$ is a Hurewicz fibration. This is true thanks to Lemma \ref{L:basechange1}.

\subsection{Gerbes}{\label{SS:gerbe}}

Let $Y$ be a topological space and $G \to Y$ a locally trivial bundle of topological groups over $Y$.
Let $\X$ be a $G$-gerbe over $Y$. Then, the structure map $p \: \X \to Y$ is  locally
a Hurewicz fibration, hence, in particular, a weak Serre fibration (Theorem \ref{T:strong}). 

To see why this is true, note that we can work locally on
$Y$, so we may assume that $G \to Y$ is a trivial bundle, i.e., is of the form $H\times X \to X$ for
some topological group $H$. By further shrinking $Y$, we may also assume that
$\X=[Y/H]$, for the trivial action of $H$ on $Y$. But in this case 
$\X=[Y/H]=[*/H]\times Y$, and the map $p\:\X \to Y$ is simply the second projection. In this case,
by Lemma \ref{L:projection}, $p$ is a Hurewicz fibration.


\section{Fiber homotopy exact sequence}{\label{S:Fiber}}

\subsection{Homotopy groups of topological stacks}{\label{SS:homotopygroups}}

There are at least two ways to define homotopy groups of a pointed topological 
stack $(\X,x)$. One is discussed in \cite{Foundations}, $\S$17. It is the standard definition
  $$\pi_n(\X,x):=[(S^n,*),(\X,x)]$$
in terms of homotopy classes of pointed maps. For this definition to make sense, 
it is argued in {\em loc.~cit.} 
that $\X$ needs to be a Serre topological stack (Definition \ref{D:HurSerre1}). However, 
it was pointed out to me by A.~Henriques that this definition makes sense for all
topological stacks thanks to Lemma \ref{L:glue}.

The second definition for the homotopy groups makes use of a classifying space $\varphi
\: X \to \X$ for $\X$ (see $\S$\ref{SS:classifying}). More precisely, we define
        $$\pi_n(\X,x):=\pi_n(X,\tilde{x}),$$
where $\tilde{x}$ is a lift of $x$ to $X$. It is shown in (\cite{Homotopytypes}, $\S$10) that
this is well-defined up to canonical isomorphism. In fact, it is shown in
(\cite{Homotopytypes}, Theorem 10.5) that
this definition is equivalent to the  previous definition whenever $\X$  is Serre. Thanks to
Lemma \ref{L:glue}, the assumption on $\X$ being Serre is redundant and the two 
definitions are indeed equivalent for all topological stacks $\X$, as we will
see in Corollary \ref{C:equal}.

 \begin{lem}{\label{L:NDR}}
  Let $A \hookrightarrow B$ be a cofibration and $a \in A$ a point.
  Let $(\X,x)$ be a pointed topological stack. Then, the natural map
    $$[(B/A,a),(\X,x)] \to [(B,A,a),(\X,x,x)]$$
  is a bijection. Here, $[-,-]$ stands for homotopy classes
  of maps (of pairs or triples).
 \end{lem}

 \begin{proof} 
  The point is that, given  a map of triples
  $(B,A,a)\to(\X,x,x)$, instead of trying to produce a
  map $B/A \to \X$, which may actually not exist,  we can
  construct a map of triples $(CA\vee_A B,CA,a) \to (\X,x,x)$,
  uniquely up to homotopy. (Here,
  $CA$ stands for the 
  cone of $A$.)
  This is true thanks  to Proposition \ref{P:pushout2}.
  It follows that the natural map
    $$[(CA\vee_A B,CA,a),(\X,x,x)] \to [(B,A,a),(\X,x,x)]$$
  is a bijection.

  Since $A \hookrightarrow B$ is a cofibration, the quotient  map
  $(CA\vee_A B,CA,a)\to (B/A,a,a)$ is a  homotopy equivalence
  of triples. So,
   $$[(B/A,a,a),(\X,x,x)] \cong [(CA\vee_A B,CA,a),(\X,x,x)]\cong  [(B/A,a),(\X,x)].$$
  This completes the proof.
 \end{proof}

\subsection{Fiber homotopy exact sequence}{\label{SS:fhes}}

In this subsection we prove that the two definitions given in $\S$\ref{SS:homotopygroups}
for homotopy groups of topological stacks agree and they enjoy  fiber homotopy exact 
sequences for weak Serre fibrations. 
 
 \begin{thm}{\label{T:fhes}}
  Let $p \: \X \to \Y$ be a weak Serre fibration of topological stacks.
  Let $x \: * \to \X$ be a point in $\X$, and $\F:=*\times_{\Y}\X$
  the fiber of $p$ over $y:=p(x)$. Then, there is a fiber homotopy exact sequence
      $$\cdots\to\pi_{n+1}(\Y,y)\to\pi_n(\F,x)\to\pi_n(\X,x)
          \to\pi_n(\Y,y)\to\pi_{n-1}(\F,x)\to\cdots.$$
  Here, $\pi_n(\X,x)$ stands for either of the two definitions of 
  homotopy groups given in $\S$\ref{SS:homotopygroups}.
 \end{thm}

 \begin{proof}
   For the first definition of $\pi_n$ (using homotopy classes of maps from $S^n$)
   the classical proof carries over, except that one has to be careful
   that to give a map of pairs $(\mathbb{D}^n,\partial\mathbb{D}^n) \to (\X,x)$ is
   not the same thing as giving a pointed map
   $(\mathbb{D}^n/\partial\mathbb{D}^n,*) \to (\X,x)$.
   This, however, is fine if we work up to homotopy, thanks to Lemma \ref{L:NDR}.

   For the second definition of $\pi_n$ we can assume, by making a base extension 
   along a universal weak equivalence $Y \to \Y$,
   that $\Y=Y$ is a topological space.
   Choose a map $X \to \X$ which is a universal weak equivalence and a
   weak Serre fibration (see Theorem \ref{T:nice} and Corollary \ref{C:weakSerre}).
   It is enough to prove the statement for the
   composite map $p' \: X \to Y$. But $p'$ is a weak Serre fibration of topological spaces by
   Lemma \ref{L:compose2}, and the claim in this case is standard.
 \end{proof}

 \begin{cor}{\label{C:equal}}
  The two definitions for $\pi_n(\X,x)$ given in $\S$\ref{SS:homotopygroups} coincide.
 \end{cor}

 \begin{proof}
  Use Corollary \ref{C:weakSerre} and  Theorem \ref{T:fhes}.
 \end{proof}

 \begin{prop}{\label{P:trivial}}
  Let $p \: \X \to \Y$ be a Serre morphism of topological stacks (Definition \ref{D:HurSerre2}).
  Then, $p$ is a trivial  Serre fibration if and only if it is a  Serre fibration
  and a weak equivalence (i.e., induces isomorphisms on all homotopy groups).
 \end{prop}

 \begin{proof}
   Assume that $p$ is a  Serre fibration and a weak equivalence. 
   Pick a classifying space $\varphi \: X \to \X$ for $X$. By 
   Corollary \ref{C:weakSerre}, $f:=p\circ \varphi$ is a weak Serre fibration and a weak equivalence.
   Therefore, the base extension $f_T$ of $f$ along an arbitrary map $T \to \Y$ from 
   a topological space $T$ is
   weak Serre fibration of topological spaces and has contractible fibers 
   (by Theorem \ref{T:fhes}). Thus, $f_T$  is a weak trivial Serre fibration.
   Since $T \to \Y$ is arbitrary, it follows from Lemma 
   \ref{L:basechange1} that  $f=p\circ \varphi$  is a weak
   trivial Serre fibration.  Lemma \ref{L:compose3}  implies that $p$ itself is  a 
   weak  trivial Serre fibration.
   Since $p$ is also a Serre fibration, it follows from Corollary \ref{C:weakTriv}
   that it is a trivial Serre fibration.

   To prove the converse, first we consider the case where $p$ is representable.
   To show that $p$ induces isomorphisms on homotopy
   groups, it is enough to show that the fibers of $p$ have vanishing
   homotopy groups (Theorem \ref{T:fhes}). But this is obvious
   because, by Lemma \ref{L:basechange1}, the base extension of $p$ along every map
   $y \: * \to \Y$ is a trivial Serre fibration. 

   Now, let $p$ be  arbitrary. Pick a classifying space $\varphi \: X \to \X$ for $X$.
   It follows from Corollary \ref{C:weakSerre} and Lemma \ref{L:compose2} that
   $p\circ \varphi$ is a weak trivial Serre fibration. Therefore, by the representable case,
   $p\circ \varphi$ is a weak equivalence. Since $\varphi$ is  also a weak equivalence,
   it follows that $p$ is a weak equivalence.
 \end{proof}

\subsection{Some examples of fiber homotopy exact sequence}{\label{SS:fhesEx}}

Let us work out  the fiber homotopy exact sequences coming from
the examples discussed in $\S$\ref{S:Examples}.

 \begin{ex}{\label{E:mapping}}
  Let $\X$ be a topological stack and $x \: * \to \X$ a point in it. Let
  $P_*\X$ be the stack of paths in $\X$ initiating at $x$, and consider
  $\ev_1 \: P_*\X \to \X$. Equivalently, $\ev_1$ is the base extension of
  the fibration $(\ev_0,\ev_1) \: P\X \to \X\times\X$
  (see Example \ref{E:fib}) along the map $(x,\id_{\X}) \: \X \to \X\times \X$.
  The fiber of $\ev_1 \: P_*\X \to \X$ over $x$ is the based loop stack $\Omega_x\X$,
  that is, we have a   fiber sequence
          $$\Omega_x\X \to P_*\X \to \X.$$
  Notice that  $P_*\X$, being a fiber of the shrinkable map (Lemma \ref{L:contractible})
  $\ev_0 \: P\X \to \X$,  is a contractible stack (Lemma \ref{L:locallyshrinkable1}).
  Therefore, applying the fiber homotopy exact sequence to the above fibration, we find
  that $\pi_n(\Omega_x\X)\cong\pi_{n+1}\X$, for $n\geq 0$.
 \end{ex}
 
 \begin{ex}{\label{E:quotientfhes}}
  Let $\X=[X/R]$ be the quotient stack of a topological groupoid $[R\sst{} X]$
  such that $s \: R \to X$ is a Hurewicz fibration. Suppose for simplicity that $X$ is
  connected. Let $x$ be a point of $X$ and  $F$ the fiber of $s$ over  $x$.
  Then, we have a fiber sequence
       $$F \to X \to \X.$$
  Applying the fiber   homotopy exact sequence to this, we obtain an exact sequence
       $$\cdots\to\pi_{n+1}\X\to\pi_nF\to\pi_nX
                 \to\pi_n\X\to\pi_{n-1}F\to\cdots.$$  
  In the special case where $\X= [X/G]$, this gives
       $$\cdots\to\pi_{n+1}\X\to\pi_nG\to\pi_nX
                 \to\pi_n\X\to\pi_{n-1}G\to\cdots.$$    
  The map $\pi_nG \to \pi_nX$ is induced by the  inclusion of the orbit $G\cdot x \hookrightarrow X$.     
 \end{ex}

 \begin{ex}{\label{E:covering}}
  Let $p \: \X \to \Y$ be a covering map of topological stacks. Then, since the fiber $F$
  of $p$ is discrete, the fiber homotopy exact sequence implies that
  $\pi_n\X\to\pi_n\Y$ is an isomorphism, 
  for all $n\geq 2$. Furthermore, $\pi_1\X \to \pi_1\Y$ is
  injective and its cokernel is in a natural bijection with $F$ (upon fixing base points).
 \end{ex}

 \begin{ex}{\label{E:gerbe}}
  Let $G \to Y$ be a locally trivial bundle of groups over a topological space $Y$. Let
  $\X$ be a $G$-gerbe over $Y$. Then, for a point $y \in Y$, we have a fiber sequence
       $$[*/H] \to \X \to Y,$$
  where we have denoted the fiber $G_y$ of $G \to Y$ by $H$.
  This gives rise to the exact sequence
        $$\cdots\to\pi_{n+1}Y\to\pi_{n-1}H\to\pi_n\X
          \to\pi_nY\to\pi_{n-2}H\to\pi_{n-1}\X\to\cdots.$$   
  Here, we have used the fact that $\pi_n[*/H]\cong\pi_{n-1}H$ for $n\geq 1$ and
  $\pi_0[*/H]=\{*\}$   (see Example \ref{E:quotientfhes}).   
 \end{ex}
\subsection{Homotopy groups of the coarse moduli space}{\label{SS:coarse}}

Any topological stack $\X$ has a coarse moduli space $\X_{mod}$ which
is an honest topological space, and there is a natural map $f\: \X \to \X_{mod}$
(see \cite{Foundations}, $\S$4.3). For example, when $\X$ is the quotient stack  $[X/R]$
of a topological groupoid $[R\sst{}X]$, then  $\X_{mod}$ is simply the
coarse quotient  $X/R$. When $\X=[X/G]$, this is just $X/G$.

We have induced maps 
   $$\pi_n(\X,x) \to \pi_n(\X_{mod},x)$$
on homotopies. For $n=0$ this is a bijection. For $n\geq 1$,  
these homomorphisms are in general far from being isomorphisms. Except in the case where $\X$
is a $G$-gerbe over $\X_{mod}$ (as in $\S$\ref{SS:gerbe}), the relation between
$\pi_n(\X,x)$ and   $\pi_n(\X_{mod},x)$ is unclear. In the case $n=1$, however, the map
      $f_* \: \pi_1(\X,x) \to \pi_1(\X_{mod},x)$
is quite interesting and there is an explicit description
for it. More precisely, under a certain not too restrictive condition on $\X$,
$\pi_1(\X_{mod},x)$ is obtained from $\pi_1(\X,x)$ by killing the images
of all inertia groups of $\X$. Combined with the computations of $\S$\ref{SS:fhesEx},
this leads to interesting formulas for the fundamental groups of coarse quotients of
topological groupoids. For more on this, we refer the reader to \cite{Slice}.

\section{Homotopy fiber of a morphism of topological stacks}{\label{S:Factor}}

In this section, we prove that every morphism of stacks has a natural fibrant replacement
(Theorem \ref{T:replacement}). We then use this to define the homotopy fiber of
a morphism of (topological) stacks (Definition \ref{D:hfib}).

Let $f \: \X \to \Y$ be a morphism of stacks. Set $\tX:=\X\times_{f,\Y,\ev_0}P\Y$, where
$P\Y=\Map([0,1],\Y)$ is the  path stack of $\X$, and $\ev_0 \: P\Y \to \Y$
is the time $t=0$ evaluation map (\cite{Mapping}, $\S$5.2). Note that if $\X$ and $\Y$ 
are topological stacks, then so is $\tX$. 
We define $p_f \: \tX \to \Y$ to
be the composition $\ev_1\circ\pr_2$, and $i_f \: \X \to \tX$ to be
the map whose first and second components are $\id_{\X}$ and
$c_{\Y}\circ f$, respectively. Here, $c_{\Y} \: \Y \to P\Y$ is the map parametrizing the
constant paths (i.e., $c_{\Y}$ is the map induced from
$[0,1] \to *$ by the functoriality of the mapping stack).

 \begin{thm}{\label{T:replacement}}
  Notation being as above, we have a factorization $f=p_f\circ i_f$,
      $$\xymatrix@M=6pt{\X \ar[r]^{i_f} & \tilde{\X} \ar[r]^{p_f} \ar@/^/[l]^{r_f}  & \Y,  }$$
  such that:
  \begin{itemize}
   \item[i)] the map $p_f \: \tX \to \Y$ is a Hurewicz
     fibration;

   \item[ii)] the map  $r_f \: \tX \to \X$ is shrinkable (Definition \ref{D:retraction}) onto the section
     $i_f \: \X \to \tX$. Here, $r_f \: \X\times_{f,\Y,\ev_0}P\Y \to \X$ is the first projection map.
  \end{itemize}
 \end{thm}

 \begin{proof} To prove (i), note that  $\tX$ sits in the following 2-cartesian diagram:
        $$\xymatrix@M=6pt{\tX \ar[d]_{(r_f,p_f)} \ar[r] & P\Y \ar[d]^{(\ev_0,\ev_1)} \\
             \X\times \Y \ar[r]_{f\times\id_Y}  & \Y\times\Y}$$
   We saw in $\S$\ref{SS:mappingfiberations} that $(\ev_0,\ev_1) \:
   P\Y \to \Y\times\Y$ is a  Hurewicz
   fibration. Hence,  $(r_f,p_f) \: \tX \to \X\times\Y$ is a Hurewicz fibration. 
   It follows from Lemma \ref{L:projection} that $p_f \: \tX \to \Y$ is also a Hurewicz
   fibration.
   
   Let us now prove (ii). By Lemma \ref{L:contractible},   $\ev_0 \: P\Y \to \Y$ 
   is  shrinkable onto $c_{\Y} \: \Y \to P\Y$. It follows from the 2-cartesian diagram
          $$\xymatrix@M=6pt{\tX \ar[d]^{r_f} \ar[r] & P\Y \ar[d]^{ev_0} \\
       \X \ar[r]_{f}  \ar@/^/ @{..>} [u]^{i_f} & \Y \ar@/^/ @{..>}  [u]^{c_{\Y}}}$$
   and Lemma \ref{L:locallyshrinkable1} that  $r_f \: \tX \to \X$ is shrinkable  onto  
   $i_f \: \X \to \tX$.
 \end{proof}

Contrary to the classical case, the map $i_f \: \X \to \tX$ is not necessarily an embedding.
This is because $c_{\Y} \: \Y \to P\Y$ is not necessarily an embedding. 
Proposition \ref{P:nonembedding} explains what  goes wrong.

 \begin{lem}{\label{L:rep}}
   Let $f \: \X \to \Y$ and $g\: \Y \to \Z$ be morphisms of stacks. Suppose that $\Y$ 
   has representable  diagonal. If $g\circ f$ is representable, then so is $f$.
 \end{lem}

 \begin{proof}
   First consider the 2-cartesian diagram
        $$\xymatrix@M=6pt{\X \ar[r]^(0.4){(\id_{\X},f)}  
              \ar[d]_f & \X\times\Y \ar[d]^{(f,\id_{\Y})} \\
                       \Y \ar[r]_(0.4){\Delta} & \Y\times\Y }$$
   Since $\Delta$ is representable, it follows that $(\id_{\X},f)$   is also representable.
   Now consider the 2-cartesian diagram
        $$\xymatrix@M=6pt{\X\times_{\Z}\Y \ar[r]^(0.6){\pr_2}  \ar[d] & \Y \ar[d]^g \\
          \X \ar[r]_{g\circ f} \ar[ru]^f \ar@/^/ [u]^{(\id_{\X},f)}& \Z   }$$
    Since $f$ is the composition of two representable morphisms $(\id_{\X},f)$ and $\pr_2$, it is
    representable itself.
 \end{proof}

 \begin{prop}{\label{P:nonembedding}}
   Let  $X$ be a  topological space and $\Y$ 
   a stack with representable diagonal (e.g., a topological stack). Let $c_{\Y} \: \Y \to \Map(X,\Y)$ 
   be the map parametrizing the constant maps. Then, $c_{\Y}$ is representable.
   If $X$ is connected, then, for every point $y$ in $\Y$,
   the fiber of $c_{\Y}$ over the point  in $\Map(X,\Y)$ 
   corresponding to the constant map
   at $y$ is homeomorphic to the space $\Map_*(X,I_y)$ of pointed maps
   from $X$ to the inertia group $I_y$ (where we have fixed
   a base point in $X$). 
   (Note that, since $\Y$ has representable diagonal, the inertia stack is representable over $\Y$, so
   the groups $I_y$ are naturally topological groups.)
 \end{prop}

 \begin{proof}
   By (\cite{Mapping}, Lemma 4.1), $\Map(X,\Y)$  has representable diagonal.
   Since $c_{\Y}$ has a
   left inverse, it follows from Lemma \ref{L:rep} that $c_{\Y}$ is representable.
   
   By simply writing down the definition of the fiber of a moprhism over a point,
   it follows that the fiber of $c_{\Y}$ over the point in $ \Map(X,\Y)$ 
   corresponding to the constant map $y$
   is equal to the quotient of $\Map(X,I_y)$ by the subgroup of constant
   maps $X \to I_y$. This quotient is homeomorphic to  $\Map_*(X,I_y)$.
 \end{proof}

 \begin{cor}{\label{C:embedding}}
    The map $i_f \: \X \to \tX$  of Theorem \ref{T:replacement} is representable.
 \end{cor}

 \begin{proof}
  Follows immediately from Proposition \ref{P:nonembedding}.
 \end{proof}

 \begin{defn}{\label{D:hfib}}
  Let $f \: \X \to \Y$ and $g \: \Z \to \Y$ be  morphisms   of  stacks.
  We define the {\bf homotopy fiber product} of $\X$ and $\Z$ over $\Y$
  to be   
       $$\X\times_{\Y}^h\Z:=\tX\times_{p_f,\Y,g}\Z=
           \X\times_{f,\Y,\ev_0}P\Y\times_{\ev_1,\Y,g}\Z,$$ 
  where $p_f$ is as in Theorem \ref{T:replacement}.   
  For a point $y\: * \to \Y$ in $\Y$, we define the {\bf homotopy fiber} of $f$ over  $y$   to be
       $$\hFib_y(f):= \tX\times_{p_f,\Y,y}*.$$
 \end{defn}
 
Since the 2-category of topological stacks is closed under fiber products, 
the homotopy fiber product  of topological stacks (and, in particular, the homotopy fiber
of a morphism of topological stacks)  is again
a topological  stack. There is a natural map 
      $$j_y \: \Fib_y(f) \to \hFib_y(f),$$ 
where $\Fib_y(f):=\X\times_{f,\Y,y}*$ is the fiber of $f$ over $y$. 
It follows from Theorem \ref{T:fhes} and the observation in the next paragraph  that
if $f$ is a weak Serre fibration, then $j_y$ is a weak  equivalence.

Since $p_f$ is a Hurewicz (hence Serre) fibration,  Theorem \ref{T:fhes}
gives a long exact sequence
    $$\cdots\to\pi_{n+1}\X\to\pi_{n+1}\Y\to\pi_n\hFib_y(f)\to\pi_n\X
           \to\pi_n\Y\to\pi_{n-1}\hFib_y(f)\to\cdots$$
on homotopy groups. Here, we have fixed
a base point $x \: * \to \X$ in $\X$, and set $y=f(x)$. (Note that $x$  gives
rise to a base point in $\Fib_y(f)$, and also to one in $\hFib_y(f)$ through the map $j_y$.) 

The construction of the fibration replacement (Theorem \ref{T:replacement}) is functorial
in the map $f$, in the sense that, given a 2-commutative square
      $$\xymatrix@M=6pt{\X' \ar[d]_{f'} \ar[r] & \X \ar[d]^f \\
       \Y' \ar[r] & \Y
     }$$
we get an induced map on the corresponding replacements, 2-commuting with all the relevant data.
It follows that the homotopy fiber  $\hFib_y(f)$ is also functorial (upon fixing a base point $y'$ in 
$\Y'$, and   letting $y$ be its image in $\Y$). In particular, the resulting fiber homotopy exact
sequences for $f$ and $f'$ are also functorial. 

The following lemma follows immediately from the fiber homotopy exact sequnece.

 \begin{lem}{\label{L:hfib}}
  In the 2-commutative square  above, if the vertical maps
  are weak equivalences, then so is the induced map
           $\hFib_y(f') \to \hFib_y(f)$
  on homotopy fibers.  
 \end{lem}

More generally, the homotopy fiber product $\X\times_{\Y}^h\Z$ is functorial in the diagram
         $$\xymatrix@M=6pt{ & \X \ar[d]^f \\
                   \Z \ar[r]_g & \Y    }$$
in the sense that if we have a morphism of diagrams given by maps
$u \: \X' \to \X$,  $v \: \Y' \to \Y$ and $w \: \Z' \to \Z$, then we have an induced
map 
        $$\X'\times_{\Y'}^h\Z' \to \X\times_{\Y}^h\Z$$
on the homotopy fiber products.

 \begin{lem}{\label{L:hfibprod}}
  In the above situation, if $u$, $v$ and $w$ are weak equivalences, then
  so is the induced map $\X'\times_{\Y'}^h\Z' \to \X\times_{\Y}^h\Z$.
 \end{lem}

 \begin{proof}
     Fix a point $y$ in $\Y$.
     Let  $q \:  \X\times_{\Y}^h\Z \to \Y$ be $g\circ\pr_2$. Using a standard argument 
     involving composing paths (for which we will need   Lemma \ref{L:glue}) we can construct
     a pair of inverse homotopy equivalences between
     $\hFib_y(f)\times\hFib_y(g)$ and $\hFib_y(q)$. The claim now follows from
     Lemma \ref{L:hfib} and the fiber homotopy exact sequence for  $q$.
 \end{proof}

\section{Leray-Serre spectral sequence}{\label{S:SS}}

In this section we use the results of the previous
sections to construct the Leray-Serre spectral sequence for a fibration of topological
stacks. As we will see shortly, there is nothing deep about the construction
of the stack version of the Leray-Serre spectral sequence, as the real work has
been done in the construction of the space version.  
All we do is to reduce the problem to the case of spaces by making careful use of 
classifying spaces for  stacks. The same method can be used to 
construct stack versions of other variants of the Leray-Serre spectral sequence as well.

\subsection{Local coefficients on topological stacks}{\label{SS:local}

Let $\X$ be a topological stack. The fundamental groupoid $\Pi\X$
of $\X$ is defined in the usual way. The objects of $\Pi\X$ are points $x \: * \to \X$. An arrow
in $\Pi\X$ from $x$ to $x'$ is a path  from $x$ to $x'$, up to
a homotopy relative to the end points. By Lemma \ref{L:glue}, $\Pi\X$ is a groupoid.

The fundamental groupoid $\Pi\X$ is functorial in $\X$.
If $\varphi \: X \to \X$ is a classifying space for $\X$, then the induced map $\Pi\varphi \:
\Pi X\to\Pi\X$ is an equivalence of groupoids.

 \begin{defn}{\label{D:localcoefficients}}
  By a {\bf system of local coefficients} on $\X$ we mean 
  a presheaf $\G$ of groups on the fundamental groupoid $\Pi\X$. In other words, a
  system of coefficients is a rule which assigns to a point $x$ in $\X$ a group $\G_x$,
  and to a homotopy class $\gamma$ of paths from $x$ to $x'$ a homomorphism
  $\gamma^*\: \G_{x'} \to \G_x$. We require that under this assignment composition of
  paths goes to composition of homomorphisms.
 \end{defn}

If $f \: \Y \to \X$ is a morphism of stacks and $\G$ a local system
of coefficients on $\X$, we obtain
a local system of coefficients $f^*\G$ on $\Y$ via the map $\Pi f \: \Pi\Y \to \Pi\X$.
If $\varphi \: X \to \X$ is a classifying space for $\X$, this pullback construction
induces a bijection (more precisely, an equivalence of categories) between local
systems on $\X$ and those on $X$.

Given a local system $\A$ of abelian groups on a topological stack $\X$, we can define {\em
homology $H_*(\X,\A)$ with coefficients in $\A$} exactly as in  (\cite{Homotopytypes}, $\S$11).
Namely, we choose a classifying space $\varphi \: X \to \X$ and set
      $$H_*(\X,\A) := H_*(X,\varphi^*\A).$$
Analogously, we can define {\em cohomology $H^*(\X,\A)$ with coefficients in $\A$}.

A very important example of a local system of coefficients is the following.
Let $f \: \X \to \Y$ be a morphism of stacks. To each point $y$ in $\Y$, associate its homotopy 
fiber $\hFib_y(f)$. For simplicity of notation, we denote $\hFib_y(f)$ by $\hFib_y$.
Given a path $\gamma$ from $y_0$ to $y_1$, we obtain a map
$e_{\gamma} \: \hFib_{y_0} \to \hFib_{y_1}$ defined by composing the path
that appears in the definition of  $\hFib_{y_0}$ with $\gamma$
(Lemma \ref{L:glue}). A standard argument shows that, given two composable
paths $\gamma$ and $\gamma'$, $e_{\gamma'}\circ e_{\gamma}$ is homotopic to
$e_{\gamma\gamma'}$. (So, in particular,  $e_{\gamma}$ is a homotopy equivalence.) 
This implies that the rule
       $$y \mapsto H^*(\hFib_y,A)$$
       $$\gamma \mapsto e_{\gamma}^*$$
is a local system of coefficients on $\Y$. Here, $A$ is an abelian group and  
         $$e_{\gamma}^* \: H^*(\hFib_{y_1},A) \to H^*(\hFib_{y_0},A)$$ 
is the induced map on cohomology.
   
As we just saw,   $e_{\gamma}$ is a homotopy equivalence. Therefore, by inverting
the direction of arrows, we find a similar local system for homology with coefficients in  $A$.
   
For a fixed abelian group $A$, it follows from the functorial properties of the homotopy fiber 
(see after Definition \ref{D:hfib}) that the above local system is functorial with respect to
2-commutative diagrams
         $$\xymatrix@M=6pt{\X' \ar[d]_{f'} \ar[r]^u & \X \ar[d]^f \\
            \Y' \ar[r]_v & \Y     }$$
Let us spell out what this means for the cohomological version. Let us call the local system associated to
$f$ by $H^*_f$. Then, the functoriality of $H^*_f$ with respect to the 2-commutative
square above means that we have a $\Pi v$-equivariant morphism of presheaves
$\rho \: H^*_{f} \to H^*_{f'}$, as in the diagram
          $$\xymatrix@M=6pt{ H^*_{f'} \ar@{-}[d]& H^*_{f} \ar[l]_{\rho} \ar@{-}[d]  \\
                 \Pi\Y' \ar[r]_{\Pi v} & \Pi\Y  }$$
(More precisely, $\rho$ is a morphism of presheaves on $\Pi\Y'$ from
the pullback of $H^*_f$ along $\Pi v$ to $H^*_{f'}$.)
It follows from Lemma \ref{L:hfib} that if $v$ and $u$ are weak equivalences, then
the above diagram is an equivalence. (More precisely, the pullback of $H^*_f$ along 
$\Pi v$ to $\Pi_{\Y}$ is equivalent to $H^*_{f'}$ via $\rho$.)
       
\subsection{Construction of the Leray-Serre spectral sequence}{\label{SS:SS}}

We now prove the existence of the Leray-Serre spectral sequence. We begin with the following.

 \begin{lem}{\label{L:replace}}
  Let $f \: \X \to \Y$ be a morphism of topological stacks. Then, there exists a 2-commutative
  square
     $$\xymatrix@M=6pt{X \ar[d]_{g} \ar[r]^{\varphi} & \X \ar[d]^f \\
        Y \ar[r]_{\psi} & \Y
     }$$
  such that  $X$ and $Y$ are topological spaces and $\varphi$ and $\psi$
  are (universal) weak equivalences. The similar statement is true when instead of a
  morphism $f \: \X \to \Y$ we start with  a diagram
       $$\xymatrix@M=6pt{ & \X \ar[d] \\
        \Z \ar[r]  & \Y
      }$$    
 \end{lem}

 \begin{proof}
     We only prove the morphism case. The proof of the diagram case is similar.
     Choose a classifying space $\psi \: Y \to \Y$, and set $\X_0=Y\times_{\Y}\X$.
     Take a classifying space $\varphi_0 \: X \to \X_0$ for $\X_0$, and set
     $\varphi := \pr_2\circ\varphi_0$ and $g:=\pr_1\circ\varphi_0$.
 \end{proof}

We now formulate the Leray-Serre spectral sequence as in \cite{McCleary}.

 \begin{thm}[{\bf the homology Leray-Serre spectral sequence}]{\label{T:HSS}}
  Let $A$ be an abelian group.
  Let $f \: \X \to \Y$ be a morphism of topological stacks with $\Y$ path-connected and
  $\F:=\hFib(f)$ connected. Then, there is a first quadrant spectral sequence $\{E_{*,*}^r,d^r\}$
  converging to $H_*(\X,A)$, with
       $$E_{p,q}^2\cong H_p\big(\Y, \HH_q(\F,A)\big),$$
  the homology of $\Y$ with local coefficients in the homology of the homotopy fiber 
  $\F$ of $f$. This spectral sequence is natural with respect to 2-commutative squares
      $$\xymatrix@M=6pt{\X \ar[d]_{f} \ar[r] & \X' \ar[d]^{f'} \\
         \Y \ar[r] & \Y'  }$$
  If $f$ is a weak Serre fibration with fiber $\F$, then  
          $$E_{p,q}^2\cong H_p(\Y, \HH_q(\F,A)).$$
 \end{thm}

 \begin{proof}
   Use Lemma \ref{L:replace} to replace $f \: \X \to \Y$ by a  continuous map
   $g \: X \to Y$ of topological spaces.  Thanks to Lemma \ref{L:hfib}, the Leray-Serre
   spectral sequence for $g$ (see \cite{McCleary}) gives rise to the desired spectral sequence for $f$.
   The functoriality (and the fact that the resulting spectral sequence is independent of the choice of
   $g$) follows from a similar reasoning as  in (\cite{Homotopytypes}, $\S$11).
   
   The last part of the theorem follows from the fact that, when $f$ is a weak Serre fibration, 
   then the natural map $j_y \: \F\to \hFib_y(f)$ is a weak equivalence (use Theorem \ref{T:fhes}). 
 \end{proof}

 \begin{thm}[{\bf the cohomology Leray-Serre spectral sequence}]{\label{T:CSS}}
  Let $R$ be a commutative ring with unit. Let $f \: \X \to \Y$ be a morphism of 
  topological stacks with $\Y$ path-connected and
  $\F:=\hFib(f)$ connected. Then, there is a first quadrant spectral sequence
  of algebras  $\{E^{*,*}_r,d_r\}$ converging to $H^*(\X,R)$ as an algebra, with
       $$E^{p,q}_2\cong H^p(\Y, \HH^q\big(\F,R)\big),$$
  the cohomology of $\Y$ with local coefficients in the cohomology of the homotopy fiber 
  $\F$ of $f$. This spectral sequence is natural with respect to 
  2-commutative squares
            $$\xymatrix@M=6pt{\X' \ar[d]_{f'} \ar[r] & \X \ar[d]^f \\
                       \Y' \ar[r] & \Y }$$
  Furthermore, the cup product $\smile$ on cohomology with local coefficients and the
  product $\cdot_2$ on $E^{*,*}_2$ are related by
  $u\cdot_2 v = (-1)^{p'q}u\smile v$, when $u\in E^{p,q}_2$ and $v\in E^{p',q'}_2$.
  If $f$ is a weak Serre fibration with fiber $\F$, then  
        $$E^{p,q}_2\cong H^p(\Y, \HH^q(\F,R)).$$
 \end{thm}

 \begin{proof}
   The same proof as Theorem \ref{T:HSS}.
 \end{proof}

\section{Eilenberg-Moore spectral sequence}{\label{S:EM}}

In this section we present the Eilenberg-Moore spectral sequence for fibrations of topological stacks.

 \begin{thm}{\label{T:EM}}
  Let $k$ be a field. Let $\Y$ be a simply-connected topological
  stack, and consider the diagram
          $$\xymatrix@M=6pt{ & \X \ar[d]^f \\
              \Z \ar[r]_g  & \Y
             }$$  
  of topological stacks. Then, there is a second quadrant spectral sequence with
        $$E_2^{*,*} \cong \Tor_{H^*(\Y,k)}\big(H^*(\X,k),H^*(\Z,k)\big)$$
  converging strongly to $H^*(\PP,k)$, where $\PP:=\X\times_{\Y}^h\Z$
  is the homotopy fiber product (Definition \ref{D:hfib}). This spectral sequence is 
  natural with respect to morphisms of diagrams.  
 \end{thm}

 \begin{proof}
  The proof is similar to the proof of Theorem \ref{T:HSS}. First we use Lemma \ref{L:replace}
  to replace the given diagram by a similar diagram of classifying spaces. Then,
  by Lemma  \ref{L:hfibprod}, the Eilenberg-Moore spectral sequence for the corresponding
  diagram of spaces (see \cite{McCleary}) gives 
  rise to the Eilenberg-Moore spectral sequence for the
  original diagram. The functoriality (and the proof that the resulting spectral sequence is 
  independent of the choice of the diagram of classifying
  spaces) follows from a similar reasoning as in (\cite{Homotopytypes}, $\S$11).
 \end{proof}


\providecommand{\bysame}{\leavevmode\hbox
to3em{\hrulefill}\thinspace}
\providecommand{\MR}{\relax\ifhmode\unskip\space\fi MR }
\providecommand{\MRhref}[2]{%
  \href{http://www.ams.org/mathscinet-getitem?mr=#1}{#2}
} \providecommand{\href}[2]{#2}

\end{document}